\documentclass[12pt]{amsart}
\usepackage{amscd}
\usepackage{amssymb,amsmath}
\usepackage{hyperref}
\usepackage{mathrsfs}
\usepackage{graphicx}
\usepackage{enumitem}
\usepackage{romannum}
\usepackage{mathtools}
\usepackage[utf8]{inputenc}

\usepackage{tikz}
\usetikzlibrary{arrows,decorations.pathmorphing,backgrounds,positioning,fit}
\usetikzlibrary{positioning}
\hypersetup{
colorlinks=true,
linkcolor=blue,}

\usetikzlibrary{trees}
\usetikzlibrary{arrows}

\def\supp{\operatorname{supp}}

\def\reg{\operatorname{reg}}

\def\max{\operatorname{max}}

\newcommand{\m}{\mathfrak m}

\newcommand{\NN}{\mathbb{N}}

\newcommand{\PP}{\mathcal{P}}

\newtheorem{lemma}{Lemma}[section]
\newtheorem{corollary}[lemma]{Corollary}
\newtheorem{theorem}[lemma]{Theorem}
\newtheorem{proposition}[lemma]{Proposition}
\newtheorem{definition}[lemma]{Definition}
\newtheorem{remark}[lemma]{Remark}
\newtheorem{example}[lemma]{Example}

\advance\headheight1.15pt

\begin{document}
\pagenumbering{arabic}
	
	\title[Regularity of powers of Cohen-Macaulay forests]{Regularity of power of edge ideal of Cohen-Macaulay weighted oriented forests} 
	
\author[Manohar Kumar]{Manohar Kumar$^*$}
\address{Department of Mathematics, Indian Institute of Technology
		Kharagpur, West Bengal, INDIA - 721302.}
\email{manhar349@gmail.com}
	
\author[Ramakrishna Nanduri]{Ramakrishna Nanduri$^{\dagger}$}
\address{Department of Mathematics, Indian Institute of Technology
		Kharagpur, West Bengal, INDIA - 721302.}
		
	\thanks{$^*$ Supported by PMRF fellowship, India}
	\thanks{$^\dagger$ Supported by SERB grant No: CRG/2021/000465, India}
	\thanks{AMS Classification 2010: 13D02, 05E99, 13D45, 13A30, 05E40}
	\email{nanduri@maths.iitkgp.ac.in}
	
	\maketitle

\begin{abstract}
	In this paper, we explicitly give combinatorial formulas for the regularity of powers of edge ideals, $\reg(I(D)^k)$, of weighted oriented unmixed forests $D$ whose leaves are sinks ($V^+(D)$ are sinks). This combinatorial formula is a piecewise linear function of $k$, for $k \geq 1$. 	
\end{abstract}
	
	\section{Introduction}
	
  A {\em weighted oriented graph} or a {\em digraph} is a graph $D=(V(D), E(D), w)$, where $V(D)$ is the vertex set of $D$, $E(D)=\{(x,y) |\mbox{ there is an edge from vertex $x$ to vertex $y$}\}$ is the {\it edge set} of $D$, and $w:V(D)\rightarrow \NN$ is a map, called weight function. That is, assign a weight $w(x)$ to each vertex $x$ of $D$. If $D$ is a weighted oriented graph, then the underlying simple graph $G$ whose 
  $V(G)=V(D)$ and $E(G)=\{\{x,y\}| (x,y)\in E(D)\}$. That is, $G$ is the simple graph without orientation and weights in $D$. Let $V(D)=\{x_1,\ldots,x_n\}$ and 
  $R=\mathbb{K}[x_1,\ldots,x_n]$, a polynomial ring in $n$ variable $x_1,\ldots,x_n$,  where $\mathbb{K}$ is a field. Then the {\it edge ideal} of $D$ is defined as 
	the ideal $$I(D)=(x_ix_j^{w(x_j)}~|~(x_i,x_j)\in E(D))\subset R.$$
  Studying the edge ideals of weighted oriented graphs is an active area of research at present times. Currently, many authors are studying various algebraic invariants of 
  $I(D)$, namely, Castelnuovo-Mumford regularity, depth, projective dimension, etc. 
  Weighted oriented graphs are important because of their applications in coding theory; see \cite{hlmrv19,prt19}. The connection between the algebraic invariants and the combinatorial invariants of a weighted oriented graph is an important problem in combinatorial Commutative Algebra. In particular, describing an algebraic invariant associated to $I(D)$ in terms of certain combinatorial information of $D$ is a challenging problem. It is well known that $\reg(I(D)^k)$ is eventually a linear function of $k$, see \cite{CHT99, K00}. It is very hard to find this linear function and very few is known in the literature.  
  
 Gimenez, Mart\'{i}nez-Bernal, Simis and Villarreal gave a combinatorial characterization of Cohen-Macaulay weighted oriented forests, see \cite[Theorem 5]{gmsv18}. 
 In \cite{prt19}, Pitones, Reyes and Toledo characterized associated primes and the unmixed property of $I(D)$. Also, see \cite{cpr22} for a characterization of unmixedness of certain weighted oriented graphs.   The projective dimension and regularity of $I(D)$ were explicitly computed by Zhu, Xu, Wang, and Tang in \cite{z19}, where $D$ is a weighted rooted forest or an oriented cycle by assuming all the weights of the vertices are at least $2$. The invariants of $I(D)^k$, where $D$ is a weighted oriented gap-free bipartite graph, have explicit combinatorial formulas provided by Zhu, Xu, Wang, and Zhang in \cite{z20}. Also see \cite{w21}. For weighted oriented rooted forests $D$, the regularity and projective dimension of $I(D)^k$ are computed in \cite{x21} by assuming all the weights of the vertices are at least $2$. For $m$-partite weighted oriented graphs, formulas for those invariants are explicitly given in \cite{z21}. In \cite{ck21}, Betti numbers of $I(D)$ are explicitly computed for a class of weighted oriented graphs $D$. Also, see \cite{mp21} to refer formulas for regularity of a class of weighted oriented graphs whose vertices have weights at least $2$. For a weighted oriented path or a cycle whose edges are oriented in one direction, then its regularity was computed in \cite{kblo22}. In this work, we proved that if $D$ is a weighted oriented Cohen-Macaulay forest with all its leaves are sinks, then we explicitly compute combinatorial formulas for regularity of $I(D)^k$ for all $k\geq 1$ and no assumptions on the weights of the vertices. This formula for $\reg(I(D)^k)$ is eventually a linear function of $k$. Our approach does not have any assumptions on the weights.          
  
  Now we give a section-wise description of the paper. In section \ref{sec2}, we recall the definitions and the basic results that we will use in the sequel. In section \ref{sec3}, we compute $\reg(I(D))$, where $D$ is a weighted oriented Cohen-Macaulay forest with all its leaves are sinks (Theorem \ref{thm1}) and we prove some technical lemmas. Finally in section \ref{sec4}, we prove a combinatorial formula for $\reg(I(D)^k)$ for any $k \geq 1$ (Theorem \ref{thm2}).    
	
\section{Preliminaries} \label{sec2}

  In this section, we recall definitions and some results which will be used throughout the paper. Also, we set up notation. \\ 
  Let $D=(V(D), E(D), w)$ be a weighted oriented graph and $V(D)=\{x_1,\ldots,x_n\}$. 
  Let $R=\mathbb{K}[x_1,\ldots,x_n]$, where $\mathbb{K}$ is a field. 
  Let $V^+(D):= \{x\in V(D): w(x)\geq 2\}$ and simple denoted by $V^+$. If $e=(x,y)\in E(D)$, then $x$ is called {\it head} of $e$ and $y$ is called {\it tail} of $e$. For a vertex $x\in V(D)$, its {\it outer neighbourhood} is defined as 
  $N_D(x)^+ := \{y \in V(D) | (x,y)\in E(D)\}$ and its {\it inner neighbourhood} is defined as $N_D(x)^- := \{z \in V(D) | (z,x)\in E(D)\}$ and $N_D[x]:=N_D(x)^+ \cup N_D(x)^- \cup \{x\}$, $N_D[x]^+:=N_D(x)^+\cup \{x\}$, $N_D[x]^-:=N_D(x)^-\cup \{x\}$. A vertex $x\in V(D)$ is called a {\it source} if $N_D(x)^-=\emptyset$ and $x$ is called a {\it sink} if $ N_D(x)^+ = \emptyset $. If $x\in V(D)$ is a source, then we set $w(x)=1$. 
  The degree of a vertex $x\in V(D)$, is defined as $d_D(x) :=|N_D(x)|$. 
  For any $S\subseteq V(D)$, we denote $D\setminus S$, the induced subgraph on the vertex set $ V(D)\setminus S $. 
  For a monomial $f\in R$, define its {\it support} as $\supp(f) := \{x_i : x_i \mid f\}$. We denote $\mathcal{G}(I)=$ minimal set of generators of $I$. For a monomial ideal $I$, its support is defined as $\supp(I):= \displaystyle \cup_{f\in \mathcal{G}(I)} \supp(f)$. 
  
  Recall that two edges $\{x,y\},\{z,u\}\in E(G)$ of a simple graph $G$ are said to be {\it adjacent} if there exists an edge between them, that is there exists an edge $\{a,b\} \in E(G)$, where $a\in \{x,y\}$ and $b\in \{z,u\}$. Two edges in a weighted oriented graph are said to be adjacent if they are adjacent in its underlying graph. We denote $[r]=\{1,\ldots,r\}$, for any positive integer $r$. Recall that a weighted oriented graph is said to be a {\it forest (or tree)} if its underlying graph is a forest(or tree). We say that a weighted oriented graph $D$ is Cohen-Macaulay if $R/I(D)$ is a Cohen-Macaulay ring.   
  
 \begin{definition} For any homogeneous ideal $I$ in $R$, define the Castelnuovo-Mumford regularity (or merely regularity) as 
 \begin{eqnarray*}
 \reg(I) &=&  \max \{j - i \mid \beta_{i,j}(I) \neq 0\} \\
         &=& \max\{j+i \mid H_{\m}^i(I)_j \neq 0\},   
 \end{eqnarray*}
  where $\beta_{i,j}(I)$ is the $(i,j)^{th}$ graded Betti number of $I$ and $H_{\m}^i(I)_j$ denotes the $j^{th}$ graded component of the $i^{t h}$ local cohomology module.  
\end{definition}
\begin{definition}
Suppose $I$ is a monomial ideal such that $I=J+K$, where $\mathcal{G}(I)= \mathcal{G}(J) \cup \mathcal{G}(K)$. Then $I=J+K$ is Betti splitting if
\begin{align*}
 \beta_{i,j}(I) = \beta_{i,j}(J) + \beta_{i,j}(J) + \beta_{i-1,j}(J \cap K)  \mbox{  for all } i,j \geq 0,
\end{align*}
where $\beta_{i-1,j}(J \cap K) = 0 $ if $i=0 .$
\end{definition}

\begin{definition}
Let $m=x_1^{a_1}\cdots x_n^{a_n}$ be a monomial in $R=k[x_1,\ldots,x_n]$. Then the polarization of $m$ is defined to be the squarefree monomial
$$\PP(m)=x_{11}x_{12} \ldots x_{1a_1} x_{21} \ldots x_{2a_2} \ldots x_{n a_n}$$ in the polynomial ring $k[x_{i j} : 1 \leq j \leq a_i,1 \leq i \leq n ].$ If $I \subset R $ is a monomial ideal with $\mathcal{G}(I)=\{m_1, \ldots, m_u\}$ and $m_i=\prod_{j=1}^{n}x_j^{a_{ij}}$ where each $a_{ij} \geq 0$ for $i=1, \ldots, m.$ Then polarization of $I$, denoted by $ I^{\PP}$, is defined as: $$I^{\PP}=(\PP(m_1), \ldots, \PP(m_u)),$$
which is a squarefree monomial ideal in the polynomial ring $R^{\PP}=k[x_{j1},x_{j2}, \ldots, x_{ja_j} \mid j=1, \ldots, n ]$ where $a_j=\max\{a_{i j} \mid i=1, \ldots,m \}$ for any $1 \leq j \leq n.$
 
\end{definition}
The following result conveys that polarization preserves some homological invariants. 
 \begin{lemma}\cite[Corollary 1.6.3]{hh11} \label{lm5}
 Let $I \subset R$ be a monomial ideal and $I^{\PP} \subset R^{\PP}$ its polarization. Then,
 \begin{enumerate}
     \item $\beta_{i,j}(I)=\beta_{i,j}(I^{\PP})$ for all $i$ and $j$,
     \item $\reg(I)=\reg(I^{\PP})$.
 \end{enumerate}
 \end{lemma}
\begin{lemma}\cite[Lemma 2.7]{x21} \label{lem5}
Let $S_1=k[x_1,\ldots,x_m]$ and $S_2=k[x_{m+1},\ldots,x_n]$ be two polynomial rings, $I \subset S_1$ and $J \subset S_2$ be two non zero homogeneous ideals. Then,
\begin{enumerate}
    \item $\reg(I+J)=\reg(I)+\reg(J)-1,$
    \item $\reg(I J)=\reg(I)+\reg(J).$
\end{enumerate}
\end{lemma}

\begin{lemma}  \cite[Corollary 12]{h14}  \label{lem: hyper}
Let $H$ be a simple hypergraph and $H'$ its induced subhypergraph. Then,
$$\reg(H') \leq  \reg(H).$$
\end{lemma}
The following result follows from Lemma \ref{lem: hyper} and \ref{lm5}.
\begin{lemma} \label{lem: induced}
Let $D'$ be an induced weighted oriented subgraph of $D$. Then,
$$\reg(I(D')^k) \leq  \reg(I(D)^k).$$
\end{lemma} 
Recall the following lemma, which we frequently use in many proofs. 
\begin{lemma}\cite[Lemma 1.2]{htt16} [Regularity lemma] \label{lem1}
		Let $0\rightarrow A \rightarrow B \rightarrow C \rightarrow 0$ be a short exact sequence finitely generated graded $R$-modules. Then 
		\begin{enumerate}
			\item $\reg(B) \leq \max\{\reg(A), \reg(C)\}$. 
			\item $\reg(A) \leq \max\{\reg(B), \reg(C)+1\}$.
			\item $\reg(C) \leq \max\{\reg(A)-1, \reg(B)\}$. 
			\item If $\reg(A) > \reg(C)+1$, then $\reg(B)=\reg(A)$. 
			\item If $\reg(C) \geq \reg(A)$, then $\reg(B)=\reg(C)$.
			\item If $\reg(A) > \reg(B)$, then $\reg(C)=\reg(A)-1$.
			\item If $\reg(B) > \reg(A)$, then $\reg(C)=\reg(B)$.
		\end{enumerate}
\end{lemma}
 The following theorem is a combinatorial characterization of Cohen-Macaulay forests. 
\begin{theorem}\cite[Theorem 5]{gmsv18} \label{thm3}
		Let $D$ be a weighted oriented forest without isolated vertices, and let $G$ be its underlying forest. 
		Then the following conditions are equivalent: 
		\begin{enumerate}
			\item $D$ is Cohen-Macaulay. 
			\item $I(D)$ is unmixed; that is, all its associated primes have the same height.
			\item $G$ has a perfect matching $\{x_1,y_1\}, \{x_2,y_2\}, \ldots, \{x_r,y_r\}$ so that 
			$d_G(y_i)=1$ for $i=1,2,\ldots,r$ and $w(x_i)=1$ if $(x_i,y_i) \in E(D)$. 
		\end{enumerate} 
\end{theorem}

\section{Castelnuovo Mumford Regularity of Cohen-Macaulay weighted oriented forests} \label{sec3}

In this section we prove a formula for $\reg(I(D))$. To show this, we prove some technical lemmas. 

\begin{lemma}\label{lem2}
Let $D$ be a weighted oriented unmixed forest. Let $\{\{x_1,y_1\}, \ldots, \{x_r,y_r\}\}$ be a perfect matching in the underlying graph $G$ of $D$. Suppose $y_1, \ldots, y_r$ are sinks. Then for any edge $(x_t,y_t)$ and $z \in V(D)$ such that $N_D(x_t)=\{y_t,z\}$ and for all $k\geq 1$, 
\begin{enumerate}
    \item $\displaystyle I(D\setminus y_t)^k \bigcap x_t y_t^{w(y_t)} I(D)^{k-1}
           = z x_t y_t^{w(y_t)}  I(D\setminus y_t)^{k-1}+ x_t y_t^{w(y_t)} I(D\setminus N_D[x_t])^k $, 
   \item $\displaystyle  z x_t y_t^{w(y_t)} I(D\setminus y_t)^{k-1} \bigcap x_t y_t^{w(y_t)}I(D\setminus N_D[x_t])^k= z x_t y_t^{w(y_t)} I(D\setminus N_D[x_t])^k$.   
 \end{enumerate}
\end{lemma}
\begin{proof}
(1) Note that $ x_t y_t^{w(y_t)} I(D\setminus N_D[x_t])^k $  and                            $ x_t y_t^{w(y_t)}z I(D\setminus y_t)^{k-1}$ are contained in $I(D\setminus y_t)^k \bigcap x_t y_t^w(y_t)I(D)^{k-1},$ because $x_t z \in \mathcal{G}(I(D\setminus y_t)).$ This implies that 
$$ z x_t y_t^{w(y_t)}  I(D\setminus y_t)^{k-1}+ x_t y_t^{w(y_t)} I(D\setminus N_D[x_t])^k \subseteq  I(D\setminus y_t)^k \bigcap x_t y_t^{w(y_t)} I(D)^{k-1}. $$
Now we prove the reverse inclusion. 
Let $f \in I(D\setminus y_t)^k \bigcap x_t y_t^{w(y_t)} I(D)^{k-1} $. Then $f=x_t y_t^{w(y_t)}f_1,$ where $f_1 \in R$ and $m \mid f,$ for some $m \in \mathcal{G}(I(D\setminus y_t)^k)$. Now We will see the proof in three further cases.  \\      
{\bf Case 1:} Suppose $x_t \mid m$. Then $z \mid m$ because $ z x_t  \in \mathcal{G}(I(D\setminus y_t)).$ In fact, we have  $ z x_t  \mid m$. Then $m= z x_t  m_1$, for some $m_1 \in \mathcal{G}(I(D\setminus y_t)^{k-1})$. Since $y_t^{w(y_t)}\mid f$, ~ $ m\mid f$ and $y_t^{w(y_t)} \nmid m$, we get that $m \mid x_t f_1$. This implies that 
$ z x_t  m_1 \mid x_t f_1$. This gives that $m_1\mid f_1$. Thus $f_1\in I(D\setminus y_t)^{k-1}$. This implies that $f \in  zx_ty_t^{w(y_t)}I(D\setminus y_t)^{k-1}.$  \\
{\bf Case 2:} Suppose $x_t \nmid m$ and $z \mid m.$ Then we have $m \in \mathcal{G}(I(D\setminus \{x_t,y_t\})^k)$. Also,  $ zu^{w(u)} \mid m $ because $zu^{w(u)} \in \mathcal{G}(I(D\setminus \{x_t,y_t\})),$ where $u \in N_D(z)\setminus x_t$ . Then $m=z u^{w(u)}m_2,$ for some $m_2 \in \mathcal{G}(I(D\setminus \{x_t,y_t\})^{k-1}).$  On the other hand, we have  $f=x_t y_t^{w(y_t)}f_1,$ where $f_1 \in R$. Thus we have $m \mid f_1$.  Therefore,
$$f=x_ty_t^{w(y_t)}f_1=x_t y_t^{w(y_t)}zu^{w(u)}m_2g_2, $$
for some  $g_2 \in R$. Thus we have $$ f \in zx_t y_t^{w(y_t)}u^{w(u)}I(D\setminus \{x_t,y_t\})^{k-1} \subseteq zx_ty_t^{w(y_t)}I(D\setminus \{x_t,y_t\})^{k-1} \subseteq zx_ty_t^{w(y_t)}I(D\setminus y_t)^{k-1}.$$  
{\bf Case 3:} Suppose $x_t \nmid m$ and $z \nmid m.$ Then $m \in \mathcal{G}(I(D\setminus N_D[x_t])^k).$ Since $f=x_ty_t^{w(y_t)}f_1,$ we get that $m \mid f_1$ because $m \mid f$ and $m \in \mathcal{G}(I(D\setminus N_D[x_t])^k).$ Therefore,  $$f=x_ty_t^{w(y_t)}f_1=x_t y_t^{w(y_t)}mg_3, $$
for some $g_3 \in R.$ This implies that $f \in x_ty_t^{w(y_t)}I(D\setminus N_D[x_t])^k.$\\

\vskip 0.1 cm
\noindent 
(2) Note that $y_t^{w(y_t)}x_tz I(D\setminus N_D[x_t])^k \subseteq y_t^{w(y_t)}x_tz I(D\setminus y_t)^{k-1} \bigcap x_t y_t^{w(y_t)}I(D\setminus N_D[x_t])^k.$ Now, we will prove the other inclusion. 
Let $f \in y_t^{w(y_t)}x_tz I(D\setminus y_t)^{k-1} \bigcap x_t y_t^{w(y_t)}I(D\setminus N_D[x_t])^k.$ Then $f=y_t^{w(y_t)}x_tzf_1 $, for some $f_1 \in R$ and $x_t y_t^{w(y_t)}m_1 \mid f$, for some $m_1 \in \mathcal{G}(I(D\setminus N_D[x_t])^k)$. This implies that  $m_1 \mid f_1.$ This gives that $f=y_t^{w(y_t)}x_tzm_1g,$ for some $g \in R.$ Thus $f \in zx_ty_t^{w(y_t)}I(D\setminus N_D[x_t])^k$, as required.  
\end{proof}

\begin{lemma}\label{lem: leaf}
Let $D$ be a weighted oriented unmixed forest. Let $\{\{x_1,y_1\}, \ldots, \{x_r,y_r\}\}$ be a perfect matching in the underlying graph $G$ of $D$. Suppose $y_1, \ldots, y_r$ are sinks. Then for any edge $(x_t,y_t)$ and $z \in V(D)$ such that $N_D(x_t)=\{y_t,z\}$, and for all $k\geq 1$, 
\begin{enumerate}
\item $(I(D \setminus y_t)^k : z x_t) = I(D \setminus y_t)^{k-1}, $
\item  $(I(D\setminus y_t)^k, x_t)=(I(D\setminus \{x_t,y_t\})^k,x_t), $
\item $((I(D\setminus y_t)^k:x_t),z)=((I(D\setminus N_D[x_t])^k:x_t), z)=(I(D\setminus N_D[x_t])^k,z).$
 \end{enumerate}
 \end{lemma}
 \begin{proof}
Note that $x_t$ is a leaf in $ D\setminus y_t $ and  $N_{D\setminus y_t}(x_t)=\{z\}$ as $N_D(x_t) = \{y_t, z\}$. Also, one can write $N_{D\setminus y_t}^{-}(x_t)=\{z\}$ and $w(x_t)=1, w(z)=1$. Thus the proof follows from \cite[Proposition 3.2]{x21}.
 \end{proof}

\begin{proposition}\label{prop3}
 	Let $D$ be a weighted oriented unmixed forest. Let  $\{\{x_1,y_1\}, \ldots, \{x_r,y_r\}\}$ be a perfect matching in the underlying graph $G$ of $D$. Suppose $y_1,\ldots,y_r$ are sinks. Then
 	 $$ \reg(I(D\setminus y_t)^k \leq \max\{\reg(I(D\setminus y_t)^{k-1})+2,\reg(I(D\setminus N_D[x_t])^k)+1,\reg(I(D\setminus \{x_t,y_t\})^k)\}.$$
 \end{proposition}
 \begin{proof}
 Let us take two exact sequences
 \begin{equation}\label{eqn1}
     0 \rightarrow (I(D\setminus y_t)^k:x_t)(-1) \xrightarrow{.x_t} I(D\setminus y_t)^k \rightarrow (I(D\setminus y_t)^k,x_t) \rightarrow 0
 \end{equation}
  \begin{equation}\label{eqn2}
     0 \rightarrow (I(D\setminus y_t)^k: z x_t)(-1) \xrightarrow{.z} (I(D\setminus y_t)^k:x_t) \rightarrow ((I(D\setminus y_t)^k:x_t),z) \rightarrow 0
 \end{equation}
The Lemma  \ref{lem: leaf}(1),  \ref{lem: leaf}(2) and  \ref{lem: leaf}(3) result in 
\begin{enumerate}
    \item $ \reg((I(D\setminus y_t)^k : zx_t)) = \reg(I(D\setminus y_t )^{k-1}), $
    \item $ \reg(I(D\setminus y_t)^k,x_t) = \reg((I(D\setminus \{x_t,y_t\})^k,z)), $ 
    \item $ \reg(((I(D\setminus y_t)^k:x_t),z)=\reg((I(D\setminus N_D[x_t])^k,z)), $
\end{enumerate}
 respectively.
Note that we have $\reg(I(D\setminus \{x_t,y_t\})^k,x_t)=\reg(I(D\setminus \{x_t,y_t\})^k)$ and $\reg(I(D\setminus N_D[x_t])^k, z)=\reg(I(D\setminus N_D[x_t])^k)$ using Lemma \ref{lem5}(1). 
 Now, we apply regularity Lemma \ref{lem1}(1) on exact sequence \eqref{eqn2}   to get  
\begin{align*}
\reg(I(D\setminus y_t)^k:x_t) & \leq \max\{\reg(I(D\setminus y_t)^{k-1})+1,\reg(I(D\setminus N_D[x_t])^k,z)\}\\
    = & \max\{\reg(I(D\setminus y_t)^{k-1})+1,\reg(I(D\setminus N_D[x_t])^k)\}.
\end{align*}
Also, we apply regularity lemma \ref{lem1}(1) on exact sequence \eqref{eqn1} to get  
 \begin{align*}
\reg(I(D\setminus y_t)^k  \leq & \max\{\reg((I(D\setminus      y_t)^k):x_t)+1,\reg((I(D\setminus \{x_t,y_t\})^k,x_t))\}\\
   = & \max\{\reg((I(D\setminus y_t)^k):x_t)+1,\reg((I(D\setminus \{x_t,y_t\})^k)\}. 
\end{align*}
 Thus $$\reg(I(D\setminus y_t)^k \leq \max\{\reg(I(D\setminus y_t)^{k-1})+2,\reg(I(D\setminus N_D[x_t])^k)+1,\reg(I(D\setminus \{x_t,y_t\})^k)\} .$$ This proves the proposition.
 \end{proof}
 \noindent
 {\bf Notation :} We fix the following notation for any weighted oriented graph $D$:  
 \begin{align*}
\Theta(k,D) = & \max \Bigg\{ (\max\{w(y_{i_j})\}+1)(k-1)+ \sum_{j=1}^{s}w(y_{i_j})+1 : \mbox{ none of the edges } \{x_{i_j},y_{i_j}\}\\
 & \mbox{ are adjacent }, y_{i_j} \in V(D)    \Bigg\}, \mbox{ for all } k \geq 1, \mbox{ and } \\
\Theta(0,D) & = 0.
\end{align*}

 \begin{lemma}\label{lem6}
	Let $D$ be a weighted oriented unmixed forest. Let  $\{\{x_1,y_1\}, \ldots, \{x_r,y_r\}\}$ be a perfect matching in the underlying graph $G$ of $D$. Suppose $y_1,\ldots,y_r$ are sinks. Then for any edge $(x_t,y_t)$ and $z \in V(D)$ such that $N_D(x_t)=\{y_t,z\}$, we have
 \begin{align*}
\Theta(k-1,D)+w(y_t)+1 \leq \Theta(k,D).
\end{align*} 
 \end{lemma}
 \begin{proof}
Suppose $\Theta(k-1,D)=(k-2)(w(y_{j_i})+1)+\sum_{n=1}^{n=s}w(y_{j_n})+1 $ for some $y_i's$ and $w(y_{j_i})= \max\{y_{j_1}, \cdots, y_{j_s} \}$.
 Then we have
\begin{align*}
\Theta(k-1,D) +w(y_t)+1= & (k-1) (w(y_{j_i}))+1)+ \sum_{n=1}^{s}w(y_{j_n})+1-(w(y_{j_i})+1)+(w(y_t)+1) \\
         = & (k-1)(w(y_{j_i}))+1) + \sum_{n=1}^{s}w(y_{j_n})+1+w(y_t)-w(y_{j_i}) \\
         \leq & \Theta(k,D)+w(y_t)-w(y_{j_i}).
\end{align*}
Suppose $ y_t \in \{y_{j_1},\ldots, y_{j_s}\} $, then we have 
$ \Theta(k-1,D)  + w(y_t)+1 \leq \Theta(k,D).$ \\
Now suppose $ y_t \not \in \{y_{j_1},\ldots, y_{j_s}\} $. Then $y_j \in  \{y_{j_1},\ldots, y_{j_s}\}$ and $$ \sum_{n=1,i_n \neq j}^{s}w(y_{j_n})+1 + w(y_j)\leq  \Theta(1, D\setminus \{N_D[x_t]\cup N_D[z]\}) +w(y_j) \leq \Theta(1, D\setminus N_D[x_t])+w(y_j) $$ for $(z,y_j) \in E(D)$ as $N_D(x_t)=\{z,y_t\}$. If possible suppose $w(y_t) > w(y_{j_i})$, then we have
\begin{align*}
  \Theta(k-1,D) = & (k-2)(w(y_{j_i})+1)+ \sum_{n=1,i_n \neq j}^{s}w(y_{j_n})+1 +w(y_j)  \\
              \leq & (k-2)(w(y_{j_i})+1)+ \Theta(1, D\setminus \{N_D[x_t]\cup N_D[z]\}) + w(y_j) \\
              < & (k-2)(w(y_t)+1)+ \Theta(1, D\setminus N_D[x_t]) + w(y_t) \mbox{ (since } w(y_t) > w(y_{j_i})  \geq w(y_j) ) \\
              \leq & \Theta(k-1,D).
\end{align*}
where last inequality holds because set of all non-adjacent pairs $ \{x_{i_j},y_{i_j}\}$ in $D\setminus N_D[x_t] $ are also non-adjacent to $\{x_t,y_t\}$ in $D$. Therefore we get contradiction that $ y_t \not \in \{y_{j_1},\ldots, y_{j_s}\} $. Hence $w(y_t)\leq w(y_{j_i}).$
This implies that $$ \Theta(k-1,D)  + w(y_t)+1 \leq \Theta(k,D).$$
\end{proof}

 \begin{lemma}\label{lm3.6}
 	Let $D$ be a weighted oriented unmixed forest. Let  $\{\{x_1,y_1\}, \ldots, \{x_r,y_r\}\}$ be a perfect matching in the underlying graph $G$ of $D$. Suppose $y_1,\ldots,y_r$ are sinks. Then for any edge $(x_t,y_t)$ and $z \in V(D)$ such that $N_D(x_t)=\{y_t,z\}$, we have
 \begin{align*}
\Theta(k,D)=\max\{\Theta(k,D\setminus N_D[x_t])+w(y_t),\Theta(k,D\setminus \{x_t,y_t\}),\Theta(k-1,D)+w(y_t)+1\}.
\end{align*}	
 \end{lemma}
 \begin{proof}
Note that $\Theta(k,D\setminus \{x_t,y_t\}) \leq \Theta(k,D)$. Since set of all non-adjacent pairs $\{x_{i_j},y_{i_j}\}$ in $D\setminus N_D[x_t]$ are also non-adjacent to $\{x_t,y_t\}$ in $D$, then  $$\Theta(k,D\setminus N_D[x_t])+w(y_t) \leq \Theta(k,D).$$ 
Also, from Lemma \ref{lem6}, we have 
$$ \Theta(k-1,D)+w(y_t)+1 \leq \Theta(k,D) .$$
This implies that
 \begin{align*}
   \max\{\Theta(k,D\setminus N_D[x_t])+w(y_t),\Theta(k,D\setminus \{x_t,y_t\}),\Theta(k-1,D)+w(y_t)+1\} \leq \Theta(k,D). 
 \end{align*}
Now we show the other inequality. Suppose $\Theta(k,D)=(w(y_l))+1)(k-1)+ \sum_{j=1}^{s}w(y_{i_j})+1$, where  $w(y_l)=\max\{w(y_{i_1}), \ldots, w(y_{i_s}) \}$ for some $l$.
 If none of the $y_{i_j}$ is equal to $y_t$, then   
 $$ (k-1) (w(y_l))+1) +\sum_{j=1}^{s}w(y_{i_j})+1 \leq \Theta(k,D\setminus \{x_t,y_t\}).$$ If $y_{i_{j_0}}=y_t$ for some $i_{j_0}$ and $y_l \neq y_t$, then from the non-adjacency of edges $\{x_{i_j},y_{i_j}\}$ we get  
 \begin{align*}
   (k-1)(w(y_l))+1) + \sum_{j=1}^{s}w(y_{i_j})+1 &= (k-1)(w(y_l)+1)+\sum_{j=1,j\neq t}^{s}w(y_{i_j})+1 +w(y_t)\\
    &\leq \Theta(k,D\setminus N_D[x_t])+w(y_t) .
 \end{align*}
Now assume $y_{i_{j_0}}=y_{l}=y_t$. Then 
\begin{align*}
    (w(y_l)+1)(k-1)+\sum_{j=1}^{s}w(y_{i_j})+1 & = (k-1)(w(y_t)+1)+\sum_{j=1,j\neq t}^{s}w(y_{i_j})+1 \\
    & = (k-2)(w(y_t)+1) +\sum_{j=1,j\neq t}^{s}w(y_{i_j})+1+w(y_t)+1 \\
    & \leq \Theta(k-1,D)+w(y_t)+1 .
\end{align*}
Thus we get the inequality
  \begin{align*}
\Theta(k,D)\leq \max\{\Theta(k,D\setminus N_D[x_t])+w(y_t),\Theta(k,D\setminus \{x_t,y_t\}),\Theta(k-1,D)+w(y_t)+1\}.
\end{align*}
This proves the lemma.
 \end{proof}
\noindent
Now we prove the main result of this section. 
\begin{theorem}\label{thm1}
		Let $D$ be a weighted oriented unmixed forest. Let $\{\{x_1,y_1\}, \ldots, \{x_r,y_r\}\}$ be a perfect matching in the underlying graph $G$ of $D$. Suppose $y_1,\ldots,y_r$ are sinks. Then
\begin{align*}
   	\reg(I(D))=\Theta(1,D) 
		            =& \max\Bigg\{ \sum_{j=1}^{s}w(y_{i_j})+1 ~:\mbox{ none of the edges } \{x_{i_j},y_{i_j}\} \\ &\mbox{ are adjacent } , y_{i_j} \in V(D) \Bigg \}. 
\end{align*}
\end{theorem} 
\begin{proof}
		Let $D$ be a weighted oriented unmixed forest with the 
		perfect matching \\ $\{\{x_1,y_1\},\ldots,\{x_r,y_r\}\}$. Note that $r=\frac{\mid V(D) \mid } {2}$.  We prove the theorem by induction on $r$. Let $r=2$. We have that $I(D)=(x_1x_2,x_1y_1^{w(y_1)},x_2y_2^{w(y_2)})$. Write 
		$I(D)=J+K$, where $J=(x_1x_2,x_1y_1^{w(y_1)})$,  $K=(x_2y_2^{w(y_2)})$. Note that $J\cap K=(x_1x_2y_2^{w(y_2)})$. One can easily see that $\reg(J)=w(y_1)+1, \reg(K)=w(y_2)+1$ and  
		$\reg(J\cap K) =w(y_2)+2$. Now applying the regularity lemma to the short exact sequence, 
		\begin{equation} \label{eq1}
			0 \rightarrow J \cap K \rightarrow J \oplus K \rightarrow J+K \rightarrow 0,
		\end{equation}
		we get that 
		\begin{eqnarray*}
			\reg(I(D)) &=& \reg(J+K) \\
			&\leq& \max \{\reg(J\cap K)-1, \reg(J\oplus K)\} \\
			&=& \max\{w(y_1)+1,w(y_2)+1\}.  
		\end{eqnarray*}
		The other inequality is true because $I(D)$ has minimal generators of degrees $w(y_1)+1$ and 
		$w(y_2)+1$. Thus $\reg(I(D))=\max\{w(y_1)+1,w(y_2)+1\}$. Therefore the theorem is true for $r=2$.
		Assume  $r \geq 3$. Choose an edge $(x_t,y_t)$ such that $N_D(x_t)=\{z,y_t\}$. Such a choice of an edge always exists in $D$. Write
		$I(D)=J+K$, where $J=I(D\setminus y_t )$ and $ K=(x_ty_t^{w(y_t)})$. Note that $\reg(K)=w(y_t)+1$. By induction hypothesis apply on $D\setminus N_D[x_t]$ and $D\setminus \{x_t,y_t\}$  we get that 
		\begin{align*}
		    \reg(I(D\setminus \{x_t,y_t\}))=\Theta(1,D\setminus \{x_t,y_t\}) \text{ and }   \reg(I(D\setminus N_D[x_t]))=\Theta(1,D\setminus N_D[x_t]).
		\end{align*}
From Lemma \ref{lem2}(1) we have $ J \cap K  =  x_t y_t^{w(y_t)}[I(D\setminus N_D[x_t])+ (z)].$ This implies that 
\begin{align*}
\reg(J \cap K) = & \reg((z) + I(D\setminus N_D[x_t]))+w(y_t)+1 \\
			= & \reg(I(D\setminus N_D[x_t]))+1+w(y_t)+1-1 \mbox{ ( by Lemma \ref{lem5} (1) })\\ 
			= &\Theta(1,D\setminus N_D[x_t])+w(y_t)+1.
\end{align*} 
Now from Proposition \ref{prop3} with $k=1$ and $\reg(I(D\setminus N_D[x_t])) \geq 2$ we get that \\ 
$  \reg(I(D\setminus y_t) \leq  \max\{\reg(I(D\setminus N_D[x_t]))+1,\reg(I(D\setminus \{x_t,y_t\}))\}.$ 
Thus we have
\begin{align}\label{eq6}
\reg(J) = \reg(I(D\setminus y_t) \leq& \max\{\Theta(1,D\setminus N_D[x_t])+1,\Theta(1,D\setminus \{x_t,y_t\})\}.
\end{align}
\noindent
First, we will show that $\reg(I(D))\leq \Theta(1,D)$. Consider the exact sequence \eqref{eq1} and apply Lemma \ref{lem1}(3) we get that
\begin{align*}
\reg(I(D))= &\reg(J+K) \\
          \leq & \max\{\reg(J \cap K)-1, \reg(J \oplus K)\} \\
           = & \max\{\Theta(1,D \setminus N_D[x_t])+w(y_t), \reg(J), w(y_t)+1\} \\
           \leq& \max\{\Theta(1,D \setminus N_D[x_t])+w(y_t),\Theta(1,D \setminus N_D[x_t])+1, \Theta(1,D \setminus \{x_t,y_t\}) \} \\
           & \hspace{9 cm } ( \mbox{ from equation \eqref{eq6} } ) \\ 
            = &\max\{\Theta(1,D\setminus N_D[x_t])+w(y_t), \Theta(1,D\setminus \{x_t,y_t\}) \} \\
           = & \Theta(1,D) \mbox{ ( from Lemma \ref{lm3.6} } ).
\end{align*}
This implies that $\reg(I(D)) \leq \Theta(1,D).$
Now, we will show that $\Theta(1,D) \leq \reg(I(D))$ in two cases.\\
{\bf Case $1$:} Suppose $\reg(J \oplus K) \geq \reg(J \cap K)$. Then
\begin{align*}
\Theta(1,D\setminus N_D[x_t])+w(y_t)=& \reg(I(D\setminus N_D[x_t]))+w(y_t) \\
  =&\reg(J \cap K ) \\
  \leq & \reg(J \oplus K) \\
  \leq & \max\{\reg(J),\reg(K)\} \\
  \leq & \max\{\Theta(1,D\setminus N_D[x_t])+1,\Theta(1,D\setminus \{x_t,y_t\}),2\} \\
  \leq & \Theta(1,D\setminus \{x_t,y_t\}) \leq \reg(I(D)),  
\end{align*}
where the last inequality holds by the virtue of $\Theta(1,D\setminus \{x_t,y_t\})=\reg(I(D\setminus \{x_t,y_t\})$ and Lemma \ref{lem: induced}.
Thus from Lemma \ref{lm3.6}, we have 
$$ \Theta(1,D)=\max\{\Theta(1,D\setminus N_D[x_t])+w(y_t),\Theta(1,D\setminus \{x_t,y_t\})\} \leq  \reg(I(D)). $$
This implies that $\Theta(1,D) \leq \reg(I(D))$. \\
{\bf Case $2$:} Suppose $\reg(J\oplus K) < \reg(J\cap K)$. Then 
by applying Lemma \ref{lem1}(6) on exact sequence we get that $$\reg(I(D))=\reg(J+K)=\reg(J \cap K)-1=\Theta(1,D\setminus N_D[x_t])+w(y_t).$$
Since $D\setminus\{x_t,y_t\}$ is an induced subgraph of a $D$, then from Lemma \ref{lem: induced} we have 
$$\Theta(1,D\setminus\{x_t,y_t\})=\reg(1,D\setminus \{x_t,y_t\})\leq \reg(I(D)).$$ Thus, from Lemma \ref{lm3.6} we have
$$ \Theta(1,D)=\max\{\Theta(1,D\setminus \{x_t,y_t\}), \Theta(1,D \setminus N_D[x_t]))+w_t\} \leq \reg(I(D). $$ 
This finishes the proof. 
\end{proof}

\begin{example}
	Let $D$ be a Cohen-Macaulay weighted oriented graph with all leaves are sinks as below.
\begin{figure}[!h]
    \centering
     \includegraphics[width=0.3 \textwidth]{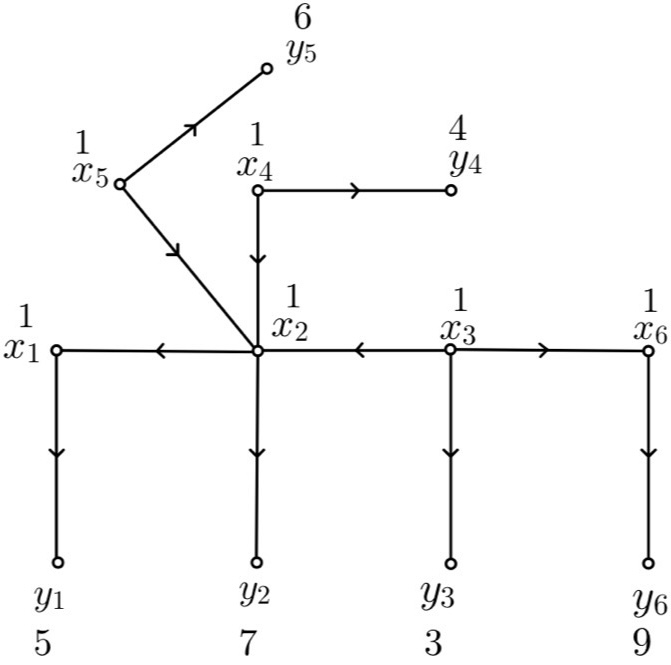} 
    \caption{Cohen-Macaulay weighted oriented forests}
    \label{fig:my_label}
\end{figure}  \\
	The weight of vertices are $w(x_i)=1$, for $1\leq i \leq 6$, $w(y_1)=5$, $w(y_2)=7$, $w(y_3)=3$, $w(y_4)=4$, $w(y_5)=6$, $w(y_6)=9$.The edge ideal of $D$ is 
	$$ I(D)=(x_1x_2, x_2x_3, x_2x_4, x_3x_6, x_2x_5, x_1y_1^5, x_2y_2^7, x_3y_3^3, x_4y_4^4, x_5y_5^6,x_6y_6^9).$$
 By the  Theorem \eqref{thm1}, we have 
\begin{eqnarray*}
\reg(I(D))&=&\max \{w(y_1)+w(y_3)+w(y_4) 
+w(y_5)+1,  w(y_1)+w(y_4) +w(y_5)+w(y_6)+1, \\ & & w(y_2)+w(y_6)+1 \} \\  
      &=& 25.  
\end{eqnarray*}
Also one can check that $\reg(I(D))=25$ by using Macaulay$2$, \cite{gs}.
\end{example}

\section{Castelnuovo-Mumford regularity of power of edge ideal of Cohen-Macaulay weighted oriented forests} \label{sec4}
In this section we prove a formula for $\reg(I(D)^k)$. To prove that we need some lemmas which are below.
 \begin{lemma}\label{lm3.8}
 Let $D$ be a weighted oriented unmixed forest. Let  $\{\{x_1,y_1\}, \ldots, \{x_r,y_r\}\}$ be a perfect matching in the underlying graph $G$ of $D$ such that $y_1,\cdots,y_r$ are sinks. Suppose $(x_t,y_t) \in E(D)$ such that $N_D(x_t)=\{y_t,z\}$ satisfying $$ \max\{\Theta(k-1,D) ,  \Theta(k,D\setminus N_D[x_t])\} \leq \Theta(k-1,D\setminus \{x_t,y_t\})+1. $$ Then 
 $$\Theta(k-1,D\setminus \{x_t,y_t\})+w(y_t)+1 \leq \Theta(k,D\setminus \{x_t,y_t\}). $$
\end{lemma}
\begin{proof}
Suppose $\Theta(k-1,D\setminus \{x_t,y_t\})=(k-2)(w(y_{i_u})+1)+ \sum_{n=1}^{s}w(y_{i_n})+1$, for some $y_{i}'s$ and  $w(y_{i_u})=\max\{w(y_{i_1}), \ldots, w(y_{i_s}) \}$.
 Then we have
\begin{align*}
\Theta(k-1,D\setminus \{x_t,y_t\}) +w(y_t)+1= & (k-2)(w(y_{i_u})+1)+ \sum_{n=1}^{s}w(y_{i_n})+1+(w(y_t)+1) \\
         = & (k-1)(w(y_{i_u})+1)+ \sum_{n=1}^{s}w(y_{i_n})+1+w(y_t)-w(y_{i_u}) \\
         \leq & \Theta(k,D)+w(y_t)-w(y_{i_u}).
\end{align*}
If possible, suppose $y_j \not \in  \{y_{i_1},\ldots, y_{i_s}\}$ for $(z,y_j) \in E(D)$ as $N_D(x_t)=\{z,y_t\}$, then 
\begin{align*}
  \Theta(k-1,D\setminus \{x_t,y_t\})=(k-2)(w(y_{i_u})+1)+ \sum_{n=1,i_n \neq j}^{s}w(y_{i_n})+1\leq  \Theta(k-1,D\setminus N_D[x_t]).  
\end{align*}
Thus from $\Theta(k-1,D\setminus N_D[x_t])+2 \leq \Theta(k,D\setminus N_D[x_t])$  $$\Theta(k-1,D\setminus \{x_t,y_t\})+1 \leq \Theta(k-1,D\setminus N_D[x_t])+1 < \Theta(k,D\setminus N_D[x_t]) $$ which contradict our assumption $ \Theta(k,D\setminus N_D[x_t]) \leq \Theta(k-1,D\setminus \{x_t,y_t\})+1. $ \\
Therefore $y_j \in  \{y_{i_1},\ldots, y_{i_s}\}$. This implies that 
$$ \sum_{n=1,i_n \neq j}^{s}w(y_{i_n})+1 + w(y_j)\leq  \Theta(1, D\setminus \{N_D[x_t]\cup N_D[z]\}) +w(y_j) \leq \Theta(1, D\setminus N_D[x_t])+w(y_j) $$ for $(z,y_j) \in E(D)$ as $N_D(x_t)=\{z,y_t\}$. If possible suppose $w(y_t) > w(y_{i_u})$, then we have
\begin{align*}
  \Theta(k-1,D\setminus \{x_t,y_t\}) + 1= & \: (k-2)(w(y_{i_u})+1)+ \sum_{n=1,i_n \neq j}^{s}w(y_{i_n})+1 +w(y_j) +1  \\
              < & \: (k-2)(w(y_t)+1)+ \Theta(1, D\setminus N_D[x_t]) + w(y_t) \\ & \hspace{5.2cm} {\mbox{ (since } w(y_t) > w(y_{i_u})  \geq w(y_j) )} \\
              \leq & \: \Theta(k-1,D)
\end{align*}
where the last inequality holds because set of all non-adjacent pairs $\{x_{i_j},y_{i_j}\}$ in $D\setminus N_D[x_t]$ are also non-adjacent to $\{x_t,y_t\}$ in $D$. Therefore we have $\Theta(k-1,D \setminus \{x_t,y_t\})+ w(y_t)+1 < \Theta(k,D)$ which contradict our assumption that
$ \Theta(k,D) \leq \Theta(k-1,D \setminus \{x_t,y_t\})+ w(y_t)+1 .$ Hence $w(y_t)\leq w(y_{i_u}).$
This implies that $$ \Theta(k-1,D\setminus \{x_t,y_t\})  + w(y_t)+1 \leq \Theta(k,D\setminus \{x_t,y_t\}).$$
\end{proof}
 
 \begin{lemma}\label{lm3.9}
Let $D$ be a weighted oriented unmixed forest. Let  $\{\{x_1,y_1\}, \cdots, \{x_r,y_r\}\}$ be a perfect matching in the underlying graph $G$ of $D$ such that $y_1,\ldots,y_r$ are sinks. Suppose $(x_t,y_t) \in E(D)$   such that $N_D(x_t)=\{y_t,z\}$ satisfying 
\begin{equation}\label{eq:hash}
    \max\{\Theta(k,D\setminus \{x_t,y_t\}), \Theta(k,D\setminus N_D[x_t])+w(y_t)\} < \Theta(k-1,D)+w(y_t)+1 \tag{\#}
\end{equation}
for some $ k \geq 1 $. Then
 $$\Theta(k-1,D)+w(y_t)+1 = (k-1)(w(y_t)+1) + \Theta(1,D\setminus N_D[x_t])+w(y_t). $$
\end{lemma}
\begin{proof}
 Suppose $\Theta(k-1,D)=(k-2)(w(y_{j_i})+1)+\sum_{n=1}^{n=s}w(y_{j_n})+1 $ for some $y_i's$ and $w(y_{j_i})= \max\{y_1, \ldots, y_r \}$. If possible, suppose $y_t \not \in \{y_{j_1}, \ldots , y_{j_s}\}$.  Then $\Theta(k-1,D) = \Theta(k-1, D\setminus \{x_t,y_t\}) .$ Note that from equation \eqref{eq:hash} we get that $\Theta(k,D\setminus N_D[x_t]) \leq \Theta(k,D) $. This implies that $$ \max\{\Theta(k-1,D\setminus \{x_t,y_t\}), \Theta(k,D\setminus N_D[x_t]) \} < \Theta(k,D\setminus \{x_t,y_t\})+1 .  $$
 Thus from Lemma \ref{lm3.8}, we get that 
 $ \Theta(k-1,D)+w(y_t)+1 \leq \Theta(k,D\setminus \{x_t,y_t\}) $ which contradicts equation \eqref{eq:hash}. Now, if possible, suppose
 $y_t \neq y_{j_i}$. Then $w(y_t) \leq w(y_{j_i})$.
 Consider,
 \begin{align*}
 \Theta(k-1,D)+w(y_t)+1 = & \: (k-2)(w(y_{j_i})+1)+\sum_{n=1}^{n=s}w(y_{j_n}) + 1                           +(w(y_t)+1) \\
                        \leq & \: (k-2)(w(y_{j_i})+1) + \sum_{n=1}^{n = s } w(y_{j_n}) + 1 + (w(y_{j_i})+1) \\
                        = & \: (k-1)(w(y_{j_i})+1) + \sum_{n=1, j_n \neq t }^{n = s} w(y_{j_n}) + 1 + w(y_t) \\
                        \leq & \: \Theta(k,D\setminus N_D[x_t])+w(y_t)
 \end{align*}
 which contradict equation \eqref{eq:hash}.  Therefore $ y_t = y_{j_i} .$ Thus 
\begin{align*}
    \Theta(k-1,D)+w(y_t)+1 = &  (k-2) (w(y_t)+1) + \sum_{n=1}^{n=s}w(y_{j_n}) + 1                                                    +(w(y_t)+1) \\
                           = &    (k-1) (w(y_t)+1) + \sum_{n=1, j_n \neq t}^{n=s}w(y_{j_n}) + 1                           + w(y_t) \\
                           = & (k-1) (w(y_t)+1) + \Theta(1,D\setminus N_D[x_t])+w(y_t) .  
\end{align*}
This finishes the proof.
\end{proof}
 
We prove the main result of this section.
\begin{theorem}\label{thm2}
	Let $D$ be an oriented unmixed forest. Let $\{\{x_1,y_1\},\ldots, \{x_r,y_r\}\}$ be a perfect matching in the underlying graph $G$ of $D$. Suppose $y_1,\ldots,y_r$ are sinks. Then for any $k\geq 1$, 
\begin{align*}
  \reg(I(D)^k)=\Theta(k,D)=&\max \Bigg\{ (k-1) (\max\{w(y_{i_j})\}+1) +  \sum_{j=1}^{s}w(y_{i_j})+1 : \mbox{ none of the} \\ & \mbox{ edges }   \{x_{i_j},y_{i_j}\}
  \mbox{ are adjacent }, y_{i_j} \in V(D)    \Bigg\}, \mbox{ for all } k \geq 1.
\end{align*}
\end{theorem} 
\begin{proof}
Let $D$ be a weighted oriented unmixed forest with the perfect matching \\ 
$\{\{x_1,y_1\},\ldots,\{x_r,y_r\}\}$. We prove the theorem by induction on $k$. If $k=1$, then the  theorem is true by Theorem \eqref{thm1}. Assume $k\geq 2$. Choose an edge $(x_t,y_t)$ such that $N_D(x_t)=\{z,y_t\}$. Such a choice of an edge always exists in $D$. Write 
\begin{eqnarray*}
 I(D)^k &=&(I(D\setminus y_t)+(x_t y_t^{w(y_t)}))^k \\
		&=& \sum_{i+j=k,i,j \geq 0,}I(D\setminus y_t)^i (x_t y_t^{w(y_t)})^j, \\
	    &=& I(D\setminus y_t)^k+ x_t y_t^{w(y_t)}\sum_{i+j=k-1,i,j \geq 0}I(D\setminus y_t)^i (x_t y_t^{w(y_t)})^j, \\
	    &=& I(D \setminus y_t)^k + x_t y_t^{w(y_t)}I(D)^{k-1}, \\
	    &=& J+L,  
\end{eqnarray*}
where $J=I(D \setminus y_t)^k$ and $L=x_t y_t^{w(y_t)}I(D)^{k-1}.$ To show the  $\Theta(k,D)=\reg(I(D)^k)$, we prove by induction on $\frac{\mid V(D) \mid}{2}$. Suppose $\frac{\mid V(D) \mid}{2} =1$. Then $D$ has only one edge $\{x_t,y_t\}.$ Then $\reg(I(D)^k)=k(w(y_t)+1)$ for all $k\geq 1$. Assume $\frac{\mid V(D) \mid}{2} \geq 2$ and the induction hypothesis.  
By induction hypothesis we have 
		\begin{align*}
		    \reg(I(D\setminus \{x_t,y_t\})^k)=\Theta(k,D\setminus \{x_t,y_t\}) \text{ and }   \reg(I(D\setminus N_D[x_t])^k)=\Theta(k,D\setminus N_D[x_t]).
		\end{align*}
By Proposition \ref{prop3} we have
\begin{align*}
    \reg(I(D\setminus y_t)^k \leq & \: \max\{\reg(I(D\setminus y_t)^{k-1})+2,
     \reg(I(D\setminus N_D[x_t])^k)+1,
      \reg(I(D\setminus \{x_t,y_t\})^k)\} \\
      \leq & \: \max\{ \reg(I(D\setminus \{x_t,y_t\})^k), \reg(I(D\setminus         \{x_t,y_t\})^{k-1}+2,\ldots,\reg(I(D\setminus \{x_t,y_t\}) \\ 
         & \: +2(k-1) ,
     \reg(I(D\setminus N_D[x_t])^k)+1,\reg(I(D\setminus N_D[x_t])^{k-1}+1+2,\ldots,  \\
     & \: \reg(I(D\setminus N_D[x_t])+1+2(k-1) \} \mbox{ (using Proposition \ref{prop3} recursively })\\
       =  & \: \max\{ \Theta(k,D\setminus \{x_t,y_t\}), \Theta(k-1,D\setminus \{x_t,y_t\})+2,\ldots,\Theta(1,D\setminus \{x_t,y_t\}) \\ 
    & \: +2(k-1) ,
     \Theta(k,D\setminus N_D[x_t])+1,\Theta(k-1,D\setminus N_D[x_t])+1+2,\ldots, \\
     & \: \Theta(1,D\setminus N_D[x_t])+1+2(k-1) \} 
\end{align*}
Thus $\Theta(k,D\setminus \{x_t,y_t\})-\Theta(k-1,D\setminus \{x_t,y_t\}) \geq 2 $,  $\Theta(k,D\setminus N_D[x_t]) - \Theta(k-1,D\setminus N_D[x_t])  \geq 2$  give that
\begin{align}\label{eq4.1}
   \reg(J)= \reg(I(D\setminus y_t)^k \leq & \max\{\Theta(k,D\setminus N_D[x_t])+1, \Theta(k,D\setminus \{x_t,y_t\}) \}.
\end{align}
Also we have $\reg(L)=\reg(x_ty_t^{w(y_t)}I(D)^{k-1})=\reg(I(D)^{k-1})+w(y_t)+1$. Then by induction
\begin{align}\label{eq4.2} 
\reg(L)=\Theta(k-1,D)+w(y_t)+1.
\end{align}
Recall that $I(D)^k = J+L$, where $J=I(D\setminus y_t)^k$ and $L=x_t y_t^{w(y_t)}I(D)^{k-1}.$
 Then we have 
 \begin{eqnarray*}
J \cap L &=& I(D\setminus y_t)^k \cap x_t y_t^{w(y_t)}I(D)^{k-1} \\
         &=& z x_t  y_t^{w(y_t)}  I(D\setminus y_t)^{k-1}+x_t y_t^{w(y_t)} I(D\setminus N_D[x_t])^k  \mbox{ ( By Lemma \ref{lem2}(1) )}\\
         &=& M+N, 
 \end{eqnarray*}
  where  $M= z x_t  y_t^{w(y_t)} I(D\setminus y)^{k-1}$  and $N=x_t y_t^{w(y_t)} I(D\setminus N_D[x_t])^k.$
  Note that
$$ \reg(N) = \reg(x_t y_t^{w(y_t)}I(D\setminus N_D[x_t])^{k})   
          = \reg(I(D\setminus N_D[x_t])^k)+w(y_t)+1  $$ 
Thus by the induction hypothesis, we have
\begin{align}\label{eq4.4}
\reg(N)= \Theta(k,D\setminus N_D[x_t])+w(y_t)+1.
\end{align}
 Now, 
 \begin{align*}
\reg(M) = & \reg( z x_t y_t^{w(y_t)}I(D\setminus y_t)^{k-1}) \\
        = & \reg(I(D\setminus y_t)^{k-1})+w(y_t)+2  \\
        \leq & \max\{\Theta(k-1,D\setminus \{x_t,y_t \}),\Theta(k-1,D\setminus N_D[x_t])+1\}+w(y_t)+2 \\
         & \hspace{10cm} \mbox{ ( from equation \eqref{eq4.1} )}.
\end{align*}
Thus $ \Theta(k,D\setminus N_D[x_t])- \Theta(k-1,D\setminus N_D[x_t]) \geq 2 $ and equation\eqref{eq4.4} give
\begin{align}\label{eq4.0}
  \reg(M) \leq \max\{\Theta(k-1,D\setminus \{x_t,y_t \}) + w(y_t) + 2, \Theta(k,D\setminus N_D[x_t]) + w(y_t) + 1 \}.  
\end{align}
Since $\Theta(k-1,D\setminus \{x_t,y_t\}) \leq \Theta(k-1,D)$ and $\Theta(k-1,D\setminus N_D[x_t])+1 \leq \Theta(k-1,D)$, then 
equation \eqref{eq4.0} implies that 
\begin{align}\label{eq4.3}
    \reg(M) \leq  \Theta(k-1,D)+w(y_t)+2 .  
\end{align}

The Lemma \ref{lem2} gives 
  $M \cap N= z x_t  y_t^{w(y_t)}  I(D\setminus N_D[x_t])^{k}$. Then we have 	
 $$ \reg(M \cap N) = \reg( z x_t  y_t^{w(y_t)}  I(D\setminus    N_D[x_t])^k) 
              =\reg(I(D\setminus N_D[x_t])^k)+w(y_t) +2. $$
\begin{align}\label{eq4.5}
\reg(M \cap N) =  \Theta(k,D\setminus N_D[x_t])+w(y_t)+2.
\end{align}              
Now consider two short exact sequences, \begin{equation} \label{eq10}
			0 \rightarrow M \cap N \rightarrow M \oplus N \rightarrow M+N \rightarrow 0,
		\end{equation}
		\begin{equation} \label{eq11}
			0 \rightarrow J \cap L \rightarrow J \oplus L \rightarrow J+L \rightarrow 0,
		\end{equation}
and apply the Lemma \ref{lem1}(3), we get that
\begin{align*}
\reg(I(D)^k)= & \reg(J+L) \\ \leq & \max\{\reg(J \cap L)-1,\reg(J \oplus L)\} \\
          \leq & \max\{\reg(M+N)-1, \reg(J), \reg(L)\}  
          \mbox{  (because } J \cap L= M + N ) \\
           \leq & \max\{\reg(M \cap N)-2, \reg(M \oplus N)-1,\reg(J),\reg(L)\} \\
          =& \max\{\reg(M \cap N)-2, \reg(M)-1, \reg(N)-1,\reg(J),\reg(L)\}
 \end{align*}
From equation \eqref{eq4.1},\eqref{eq4.2}, \eqref{eq4.4}, \eqref{eq4.3} and \eqref{eq4.5} we get that
\begin{align*}
\reg(I(D)^k)= &\reg(J+L)\\ 
         \leq & \max\{\Theta(k,D\setminus N_D[x_t]) + w(y_t),\Theta(k-1,D) + w(y_t)+1, \Theta(k,D\setminus \{x_t,y_t\}\} \\
            = & \: \Theta(k,D) \mbox{ ( from Lemma \ref{lm3.6} } ).
\end{align*}
This implies that $\reg(I(D)^k) \leq \Theta(k,D).$
 To show that $\Theta(k,D) \leq \reg(I(D)^k),$ it suffices to show that
\begin{enumerate}[label=(\roman*)]
    \item $\Theta(k,D\setminus \{x_t,y_t\}) \leq \reg(I(D)^k)$,
    \item $\Theta(k-1,D)+w(y_t)+1 \leq \reg(I(D)^k),$
    \item $\Theta(k,D\setminus N_D[x_t])+w(y_t) \leq \reg(I(D)^k),$
\end{enumerate}
by using Lemma \ref{lm3.6}.
Note that  (\romannum{1}) is always true using Lemma \ref{lem: induced}  as $ D\setminus \{x_t,y_t\} $ is an induced subgraph of $D$.  
 To show (\romannum{2}), we will take three possible cases.\\
 {\bf Case 1:} Suppose  $\reg(J \oplus L) > \reg(J \cap L).$ Then applying regularity Lemma \ref{lem1}(7) on exact sequence \eqref{eq11} we get $\reg(I(D)^k) = \reg(J \oplus L).$ Thus from equation \eqref{eq4.2} and $\reg(L) \leq \reg(J \oplus L)$ we get that $ \Theta(k-1,D)+w(y_t)+1 \leq \reg(I(D)^k).$ \\
 \noindent
{\bf Case 2:} Suppose $\reg(J \oplus L) < \reg(J \cap L).$ Then applying regularity Lemma \ref{lem1}(6)  on exact sequence \eqref{eq11} we get that $\reg(I(D)^k)=\reg(J \cap L) -1 $. Thus from equation \eqref{eq4.2} and $\reg(J \oplus L) \leq \reg(J \cap L)-1 $ we get that $\Theta(k-1,D)+w(y_t)+1 \leq \reg(I(D)^k) .$ \\
\noindent
{\bf Case 3:} Suppose $\reg(J \oplus L)=\reg(J \cap L).$ Then applying regularity Lemma \ref{lem1}(3) on exact sequence \eqref{eq11} we get that 
 \begin{align*}
   \reg(I(D)^k) \leq & \max\{\reg(J \cap L) -1, \reg(J\oplus L)  \} 
                = \reg(J \oplus L) 
 \end{align*}
 Thus from equations \eqref{eq4.1} and \eqref{eq4.2} we have 
\begin{align}\label{eq15}
     \reg(I(D)^k) \leq \max\{\Theta(k,D\setminus \{x_t,y_t\}),\Theta(k,D\setminus N_D[x_t])+ 1 , \Theta(k-1,D)+w(y_t)+1 \}.
\end{align}
 \noindent
 If $\Theta(k-1,D)+w(y_t)+1 \leq \Theta(k,D\setminus \{x_t,y_t\}),$ then we have 
 $$\Theta(k-1,D)+w(y_t)+1 \leq \Theta(k, D\setminus \{x_t,y_t\}) = \reg(I(D\setminus \{x_t,y_t\})^k) \leq \reg(I(D)^k) $$
because $D\setminus \{x_t,y_t\}$ is an induced subgraph of $D$. \\ 
Now suppose $\Theta(k,D\setminus \{x_t,y_t\}) < \Theta(k-1,D)+w(y_t)+1 $. Then we will take two subcases.
{\em Subcase 1 :}
Suppose $\reg(M \cap N) \leq \reg(M \oplus N).$ 
Then applying regularity Lemma \ref{lem1}(3) on exact sequence \eqref{eq10} we get that 
\begin{align*}
    \reg(M+N) &\leq \max\{\reg(M\cap N)-1,\reg(M \oplus N\}, \\
              &=\reg(M \oplus N)=\reg(M) \mbox{ ( since } \reg(M\cap N) > \reg(N) ).
\end{align*}
Hence, $$ \reg(J \oplus L)=\reg(J \cap L) = \reg(M+N) \leq \reg(M) $$ 
which implies that $\Theta(k-1,D)+w(y_t)+1 = \reg(L) \leq \reg(M)$ by the virtue of equation \eqref{eq4.2}. Note that from assumption $\reg(M \cap N) \leq \reg(M \oplus N)$, equations \eqref{eq4.4} and \eqref{eq4.5} we obtain $\Theta(k,D\setminus N_D[x_t])+w(y_t)+1 < \reg(M)$. 
 Therefore using equation \eqref{eq4.0}, Lemma \ref{lm3.8} and assumption $\Theta(k,D\setminus \{x_t,y_t\}) < \Theta(k-1,D)+w(y_t)+1 $  we get that
 $$\Theta(k-1,D)+w(y_t)+1 = \reg(M) = \Theta(k,D\setminus \{x_t,y_t\})+1.$$ 
which implies  
 $$ \max\{\Theta(k ,D \setminus \{x_t,y_t\}),\Theta(k,D\setminus N_D[x_t])+w(y_t)\} < \Theta(k-1,D)+w(y_t)+1.$$
\noindent
{ \em Subcase 2:}
  Suppose $\reg(M \oplus N) < \reg(M \cap N)$.  Then applying regularity Lemma \ref{lem1}(6) on exact sequence \eqref{eq10} and using equation  \eqref{eq4.5} we get that
\begin{align*}
 \reg(J \cap L) = \reg(M+N) = \reg(M\cap N)-1=\Theta(k,D\setminus N_D[x_t])+w(y_t)+1.
\end{align*}
Using  $\Theta(k,D\setminus \{x_t,y_t\}) < \Theta(k-1,D)+w(y_t)+1 $, $\reg(J\oplus L)=\reg(J \cap L)$, equations \eqref{eq4.1} and \eqref{eq4.2}, we have 
$ \Theta(k-1,D) + w(y_t) + 1 = \Theta(k,D\setminus N_D[x_t]) + w(y_t) + 1 $
which implies $\max\{\Theta(k,D \setminus \{x_t,y_t\}),\Theta(k,D\setminus N_D[x_t])+w(y_t)\} < \Theta(k-1,D)+w(y_t)+1.$ \\
Thus { \em Subcase 1 } and { \em Subcase 2 } give that
$$\max\{\Theta(k, D \setminus \{x_t,y_t\}),\Theta(k,D\setminus N_D[x_t])+w(y_t)\} <  \Theta(k-1,D)+w(y_t)+1 $$ which further results in $\reg(I(D)^k) \leq  \Theta(k-1,D)+w(y_t)+1 $, by the virtue of equation \eqref{eq15}.
Now we will show that $\reg(I(D)^k) =  \Theta(k-1,D)+w(y_t)+1 .$ \\
Write, $I(D)^k=I_1+I_2 $, where $I_2=(x_t y_t^{w(y_t)})^k$, $\mathcal{G}(I_1)=\mathcal{G}(I(D)^k)\setminus \mathcal{G}(I_2)$ and  $I_1 \cap I_2 = I_2 (I(D\setminus N_D[x_t]) + (z) )$. 
 Further write,
$(I(D)^k)^{\PP}=(I_1)^{\PP}+(I_2)^{\PP}  $, where  $(I(D)^k)^{\PP}$, $I_1^{\PP}$  and $I_2^{\PP}$ are the polarization of $I(D)^k$, $I_1$ and $I_2$ respectively. Thus $(I(D)^k)^{\PP}=(I_1)^{\PP}+(I_2)^{\PP}  $ is Betti splitting by the virtue of \cite[Corollary 2.7]{cha19}. Then from \cite[ Corollary 2.2]{cha19} we have 
$$\reg((I(D)^{\PP}))=\max\{\reg(I_1^{\PP}), \reg(I_2^{\PP}), \reg(I_1^{\PP} \cap I_2^{\PP})-1 \}.$$ 
 Further,
\begin{align*}
\reg(I_1^{\PP} \cap I_2^{\PP}) = & \reg((I_1 \cap I_2 )^{\PP})   \\
                   = & \reg(I_1 \cap I_2) \mbox{ (from Lemma } \ref{lm5}) \\
                   = & k ( w(y_t)+1 )+\Theta(1,D\setminus N_D[x_t]) \\
                   = & (k-1)(w(y_t)+1)+\Theta(1,D\setminus N_D[x_t])+w(y_t)+1 \\
                   = &  \Theta(k-1,D)+w(y_t) +2 \mbox{ (from Lemma } \ref{lm3.9}),
\end{align*}
  equation \eqref{eq15}, Lemma \ref{lm5} and  \cite[Lemma 3.1 ]{h14} result in 
$$ \reg(I_1^{\PP}) \leq \reg((I(D)^k)^{\PP})\leq  \Theta(k-1,D)+w(y_t)+1 $$ and  Lemma \ref{lm3.9} results in $\reg(I_2^{\PP}) \leq \Theta(k-1,D)+w(y_t)+1.$
Hence  using Lemma \ref{lm5} we have  $\reg(I(D)^k)= \Theta(k-1,D)+w(y_t)+1 .$ 
Thus from all the cases, we obtain 
$$ \Theta(k-1,D)+w(y_t)+1 \leq \reg(I(D^k)).$$
\noindent
Now we will prove (\romannum{3}).
Applying regularity Lemma \ref{lem1}(2) on exact sequence  \eqref{eq10}, we get that
$$ \reg(M \cap N) \leq \max\{\reg(M \oplus N), \reg(M+N)+1\}.$$
 Thus from $\reg(M \cap N)=\reg(N)+1 $ and  $ \reg(M+N)=\reg(J \cap L) $ we have
\begin{align*}
     \reg(N)+1 \leq \max\{ \reg(M), \reg(J \cap L) + 1 \}. 
\end{align*}
Now applying regularity Lemma \ref{lem1}(2) on exact sequence \eqref{eq11} we get that
\begin{align*}
         \reg(N)+1  & \leq \max\{\reg(M),\reg(J \oplus L)+1, \reg(I(D)^k)+2\} 
\end{align*}
which implies that
\begin{align*}
     \Theta(k, D\setminus N_D[x_t])+w(y_t)  \leq \max\{\Theta(k-1,D)+w(y_t), \reg(I(D)^k)\} 
\end{align*}
by virtue of equations \eqref{eq4.1}, \eqref{eq4.2}, \eqref{eq4.4} and \eqref{eq4.3}.
Thus by  (\romannum{2}), we get that
 $$ \Theta(k,D\setminus N_D[x_t])+w(y_t) \leq \reg(I(D)^k) .$$
Thus 
\begin{align*}
    \Theta(k,D) = & \max\{\Theta(k,D\setminus \{x_t,y_t\}), \Theta(k,D\setminus N_D[x_t])+w(y_t),                \Theta(k-1,D)+w(y_t)+1\} \\
               \leq & \reg(I(D^k)).
\end{align*}
This proves the theorem.
\noindent
\end{proof}

\begin{corollary}\label{corrol1}
	Let $D$ be an oriented unmixed forest. Let $\{\{x_1,y_1\},\ldots, \{x_r,y_r\}\}$ be a perfect matching in the underlying graph $G$ of $D$. Suppose $y_1,\ldots,y_r$ are sinks. Then for any $k\geq 1$, 
\begin{align*}
  \reg(I(D)^k) \leq (k-1)(w+1) + \reg(I(D))  \mbox{      for any } k \geq 1,
\end{align*}
where $w=\max\{ w(x) : x \in V(D) \}$.
\end{corollary}
\begin{proof}
Observe that $\Theta(k,D) \leq (k-1)(w+1) + \Theta(1,D) $, where $w=\max\{ w(x) : x \in V(D) \}$ and from Theorem \ref{thm1} we have 
$\reg(I(D)) =  \Theta(1,D) $. Thus the Theorem \ref{thm2} gives that
\begin{align*}
\reg(I(D)^k)=\Theta(k,D) \leq (k-1)(w+1) + \Theta(1,D)=  (k-1)(w+1) + \reg(I(D)), 
\end{align*}
 for any $k \geq 1$. Thus 
 \begin{align*}
  \reg(I(D)^k) \leq (k-1)(w+1) + \reg(I(D))  \mbox{      for any } k \geq 1. 
\end{align*}
 \end{proof}

\begin{remark}
Note that the formula for $\reg(I(D)^k)$ in the Theorem \ref{thm2} is also equal to the following symmetric formula, 
\begin{align*}
        \max \left\{(w(y_{i_l})+1) k + \sum_{j=1,j \neq l}^s w(y_{i_j}) ~:\text{ none of the edges } \{x_{i_j},y_{i_j}\} \text{ are adjacent }, y_{i_j} \in V(D)  \right \}
\end{align*}
for any $k\geq 1$. 
\end{remark}

\begin{example}
	Let $D$ be a Cohen-Macaulay weighted oriented graph with all leaves are sinks as below.
\begin{figure}[!h]
    \centering     \includegraphics[width=0.25 \textwidth]{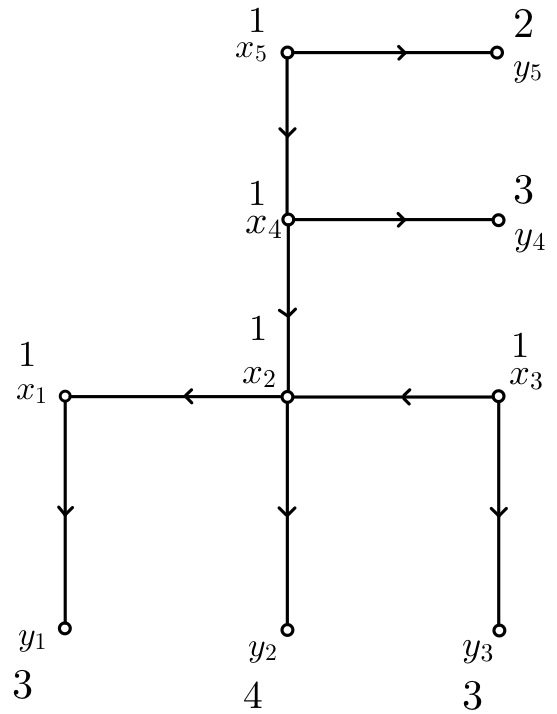} 
    \caption{Cohen-Macaulay weighted oriented forests}
    \label{fig:my_label}
\end{figure} \\
	The weight of vertices are $ w(x_i) = 1 $, for all $1\leq i \leq 5 $, $w(y_1)=3$, $w(y_2)=4$, $w(y_3)=3$, $w(y_4)=3$, $w(y_5)=2$. Then the edge ideal of $D$ is 
	$$ I(D)=(x_1x_2, x_2x_3, x_2x_4, x_4x_5, x_1y_1^3, x_2y_2^4, x_3y_3^3, x_4y_4^3, x_5y_5^2). $$
By the Theorem \eqref{thm2}, for any  $k \geq 1$,
\begin{equation*} 
\begin{split}
\reg(I(D)^k) = & \max \Bigg\{(\max\{w(y_1),w(y_3),w(y_4)\}+1)(k-1)+w(y_1)+w(y_3)+w(y_4)+1, \\    & ~~~(\max\{w(y_1),w(y_3),w(y_5)\}+1)(k-1)+ w(y_1)+w(y_3) +w(y_5)+1, \\ 
      & ~~~ (\max\{w(y_2),w(y_5)\}+1)(k-1)+ w(y_2)+w(y_5)+1 \}   \Bigg \} \\ 
      = & \max \{4(k-1)+10,~ 4(k-1)+9, ~5(k-1)+7 \}\\
      = & \max \{4(k-1)+10, ~ 5(k-1)+7\}.  
\end{split}      
\end{equation*}

This implies that $\reg(I(D)^k)=4(k-1)+10$, for $1 \leq k \leq 4$, and  $\reg(I(D)^k)=5(k-1)+10$, 
for all $k\geq 4$. 
\end{example}
The formula in Theorem \eqref{thm2}  does not hold for Cohen-Macaulay weighted oriented forests with a leaf is not sink.
\begin{example}
Let $D$ be a Cohen-Macaulay weighted oriented graph as given below.
\begin{figure}[!h]
\centering
    \includegraphics[width=0.3 \textwidth]{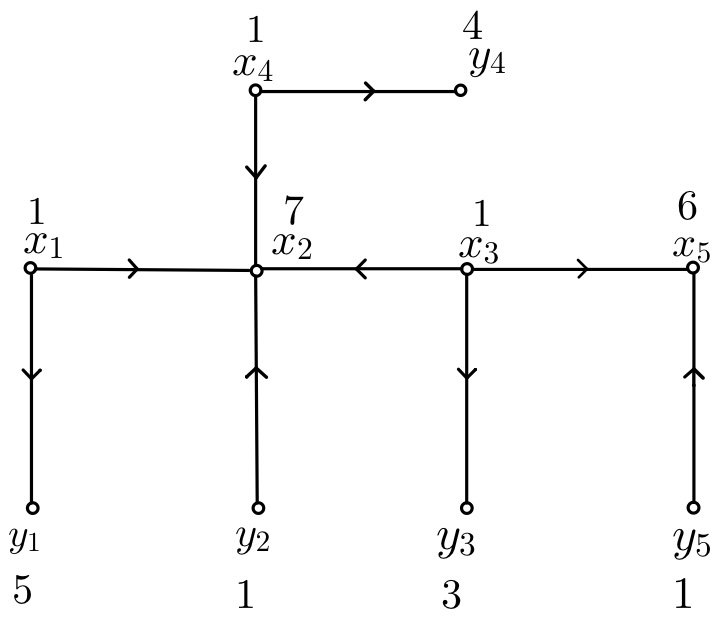} 
    \caption{Cohen-Macaulay weighted oriented forests}
    \label{fig:my_label}
\end{figure}\\
The weight of vertices are $w(x_i)=1$, for  $i=1,3,4$, $w(x_2)=7, w(x_5)=6$, $w(y_1)=5, w(y_3)=3, w(y_4)=4$ and $w(y_i)= 1$ for $i=2,5$.Then the edge ideal of $D$ is 
	$$ I(D)=(x_1x_2^7, x_2^7x_3,x_4x_2^7,x_4y_4^4, x_1y_1^4, y_2x_2^7, x_3y_3^3,y_5x_5^6,x_4y_4^4).$$
Consider for any  $k \geq 1$,
\begin{equation*} 
\begin{split}
\Theta(k,D) =& \max \Bigg\{(\max\{w(y_1),w(y_3),w(y_4)\}+1)(k-1)+w(y_1)+w(y_3)+w(y_4)+1, \\    & ~~~(\max\{w(y_1),w(y_5)\}+1)(k-1)+ w(y_1) +w(y_5)+1, \\ 
      & ~~~ (\max\{w(y_2),w(y_5)\}+1)(k-1)+ w(y_2)+w(y_5)+1 \} \\
      & ~~~ (\max\{w(y_4),w(y_5)\}+1)(k-1)+ w(y_4)+w(y_5)+1 \}     \Bigg \} \\ 
      &= \max \{6(k-1)+11,~ 6(k-1)+7, ~2(k-1)+3 , ~5(k-1)+6 \}\\
      &= 6(k-1)+11.  
\end{split}      
\end{equation*} 
This implies that $ \Theta(k,D) = 6(k-1) + 11 $, for $k \geq 1$. Now, using Macaulay $2$  \cite{gs} we have $ \reg(I(D))=24 \neq \Theta(1,D) = 11 $, $ \reg(I(D)^2)=31 \neq \Theta(2,D) = 17 $.
\end{example}
\vskip 0.3cm 
\noindent 
{\bf Acknowledgement:} 
Manohar Kumar is thankful to the Government of India for supporting him in this work through the Prime Minister Research Fellowship.
\vskip 0.3cm 
\noindent 
{\bf Data availability statement} Data sharing is not applicable to this article as no datasets were generated or analyzed during the current study.

\end{document}